\documentclass[reqno,a4paper,12pt]{amsart}

\usepackage[all,poly]{xy}
\usepackage{amsfonts,stmaryrd}
\usepackage[mathcal]{eucal}
\usepackage{amssymb}
\usepackage{amsmath}
\usepackage{mathrsfs}
\usepackage{color}
\usepackage[pagebackref,colorlinks]{hyperref}
\usepackage{enumerate}

\theoremstyle{plain}
\newtheorem*{Mainlem}{Main Lemma}
\newtheorem*{Lemma}{Lemma}
\newtheorem{The}{Theorem}
\newtheorem{Oldie}{Theorem}

\theoremstyle{remark}

\theoremstyle{definition}

\newcommand{\GL}{\operatorname{GL}}
\newcommand{\SL}{\operatorname{SL}}
\newcommand{\PGL}{\operatorname{PGL}}
\newcommand{\PE}{\operatorname{PE}}
\newcommand{\Sp}{\operatorname{Sp}}
\newcommand{\map}{\longrightarrow}
\newif\ifcomm
\let\ifcomm\iffalse

\def\a{\alpha}
\def\b{\beta}
\def\g{\gamma}

\def\A{\operatorname{A}}
\def\B{\operatorname{B}}
\def\C{\operatorname{C}}
\def\D{\operatorname{D}}
\def\F{\operatorname{F}}
\def\G{\operatorname{G}}

\def\K{\operatorname{K}}
\def\GF#1{{\mathbb F}_{\!#1}}
\def\rk{\operatorname{rk}}
\def\sr{\operatorname{sr}}

\title{Generation of relative commutator subgroups\\ in Chevalley groups. II}

\author{Nikolai Vavilov}
\address{Department of Mathematics and Mechanics,\\
St.~Petersburg State University,\\ St.~Petersburg, Russia}
\email{nikolai-vavilov@yandex.ru}
\thanks{The work of the first author is supported by the
Russian Foundation of Fundamental Research, grant 18-31-20044.}

\author{Zuhong Zhang}
\address{Department of  Mathematics\\
 Beijing Institute of Technology\\
 Beijing, China}
\email{zuhong@hotmail.com}

\date{}

\begin{document}

\begin{abstract}
In the present paper, which is a direct sequel of our paper 
\cite{RNZ-generation} joint with Roozbeh Hazrat, we prove unrelativised version of the 
standard commutator formula in the setting of Chevalley groups. 
Namely, let $\Phi$ be a reduced irreducible root system of rank 
$\ge 2$, let $R$ be a commutative ring and let $I,J$ be two ideals 
of $R$. We consider subgroups of the Chevalley group $G(\Phi,R)$
of type $\Phi$ over $R$. The unrelativised elementary subgroup 
$E(\Phi,I)$ of level $I$ is generated (as a group) by the elementary 
unipotents $x_{\a}(\xi)$, $\a\in\Phi$,  $\xi\in I$, of level $I$. 
Obviously, in general $E(\Phi,I)$ has no chances to be normal
in $E(\Phi,R)$, its normal closure in the absolute elementary subgroup $E(\Phi,R)$ is denoted by $E(\Phi,R,I)$. The main results of 
\cite{RNZ-generation} implied that the commutator 
$\big[E(\Phi,I),E(\Phi,J)]$ is in fact normal in $E(\Phi,R)$. In
the present paper we prove an unexpected result that in fact
$\big[E(\Phi,I),E(\Phi,J)]=\big[E(\Phi,R,I),E(\Phi,R,J)\big]$.
It follows that the standard commutator formula also holds
in the unrelativised form, namely 
$\big[E(\Phi,I),C(\Phi,R,J)]=\big[E(\Phi,I),E(\Phi,J)\big]$,
where $C(\Phi,R,I)$ is the full congruence subgroup  of level $I$.
In particular, $E(\Phi,I)$ is normal in 
$C(\Phi,R,I)$. 
\par\smallskip\noindent

\end{abstract}

\maketitle

\section*{Introduction}

Let $\Phi$ be a reduced irreducible root system of rank $\ge 2$,
let $R$ be a commutative ring with 1, and let $G(\Phi,R)$ be a
Chevalley group of type $\Phi$ over $R$. For the background on
Chevalley groups over rings see \cite{NV91} or \cite{vavplot},
where one can find many further references. We fix a split maximal
torus $T(\Phi,R)$ in $G(\Phi,R)$ and consider root unipotents
$x_{\alpha}(\xi)$
elementary with respect to $T(\Phi,R)$. The subgroup $E(\Phi,R)$
generated by all $x_{\alpha}(\xi)$, where $\alpha\in\Phi$,
$\xi\in R$, is called the {\it absolute\/} elementary subgroup of
$G(\Phi,R)$.
\par
Now, let $I\unlhd R$ be an ideal of $R$. Then the 
{\it unrelativised elementary subgroup\/} $E(\Phi,I)$ of level $I$ is 
defined as the subgroup of $E(\Phi,R)$, generated by all elementary 
root unipotents $x_{\alpha}(\xi)$ of level $I$,
$$ E(\Phi,I)=
{\big\langle x_{\alpha}(\xi)\mid \alpha\in\Phi,\
\xi\in R\big\rangle}. $$
\noindent
In general, this subgroup has no chances to be normal in $E(\Phi,R)$.
Its normal closure $E(\Phi,R,I)=E(\Phi,I)^{E(\Phi,R)}$ is called the
{\it relative elementary subgroup\/} of level $I$.
\par

The starting point of the present paper are the following three
observations contained in \cite{RNZ-generation}. The first one
is the left-most (non-trivial!) inclusion in Theorem 3.1, whereas the
other two are Corollary 5.2 and Corollary 5.1 of Theorem 1.3, 
respectively. In these results some additional assumptions are 
necessary in the cases $\Phi=\C_l,\G_2$. The first of these
results relies on a calculation, that is immediate for simply laced
systems, but rather non-trivial in the exceptional cases 
$\Phi=\C_2,\G_2$. The other two are easy corollaries of 
this result and the main result of \cite{RNZ-generation}, 
describing generators of the mixed commutator subgroup $[E(\Phi,R,I),E(\Phi,R,J)]$.

In the rest of this paper we impose the following umbrella assumption:

(*) In the cases $\Phi=\C_2,\G_2$ assume that $R$ does not have 
residue fields $\GF{2}$ of two elements, and in the case 
$\Phi=\C_l$, $l\ge 2$, assume additionally that any $\theta\in R$ 
is contained in the ideal $\theta^2R+2\theta R$.

This condition arises in the computation of the lower level of
$[E(\Phi,I),E(\Phi,J)]$, in \cite{RNZ2}, Lemma 17, and 
\cite{RNZ-generation}, Theorem 3.1, see also further related
results, and discussion of this condition in 
\cite{Stepanov_calculus, Stepanov_nonabelian}.

\begin{Oldie}\label{Oldie:1}
Let $\Phi$ be a reduced irreducible root system of rank $\ge 2$
and let $I,J$ be two ideals of a commutative ring $R$. Then one
has the following inclusion
$$ E(\Phi,R,IJ)\le [E(\Phi,I),E(\Phi,J)]. $$ 
\end{Oldie}

\begin{Oldie}\label{Oldie:2}
Let $\Phi$ be a reduced irreducible root system of rank $\ge 2$
and let $I,J$ be two ideals of a commutative ring $R$. Then the
mixed commutator subgroup
$[E(\Phi,I),E(\Phi,J)]$ is normal in $E(\Phi,R)$. 
\end{Oldie}

\begin{Oldie}\label{Oldie:3}
Let $\Phi$ be a reduced irreducible root system of rank $\ge 2$
and let $I,J$ be two ideals of a commutative ring $R$. Then
$$ [E(\Phi,I),E(\Phi,R,J)]=[E(\Phi,R,I),E(\Phi,R,J)]. $$
\end{Oldie}

What we have not noticed when writing \cite{RNZ-generation} is
that modulo some further elementary calculations involving our 
generators of $[E(\Phi,R,I),E(\Phi,R,J)]$, 
Theorems~\ref{Oldie:1}--\ref{Oldie:3} 
admit the following common generalisation.

\begin{The}\label{The:1}
Let $\Phi$ be a reduced irreducible root system of rank $\ge 2$
and let $I,J$ be two ideals of a commutative ring $R$. Then
$$ [E(\Phi,I),E(\Phi,J)]=[E(\Phi,R,I),E(\Phi,R,J)]. $$
\end{The}

As a matter of fact, Theorem~\ref{The:1} can be derived from 
the main result of
\cite{RNZ-generation}, Theorem 1.3. That theorem, which
we recall as Theorem~\ref{Oldie:5} in \S~\ref{sec:old}, 
lists three types of generators of
$[E(\Phi,R,I),E(\Phi,R,J)]$. Of those three types, the last two are 
contained already in $[E(\Phi,I),E(\Phi,J)]$, the second one by the 
very definition, the last one by the above Theorem~\ref{Oldie:1}. It
remains to show that the first type of generators, those of the form
$\big[x_{\alpha}(\xi),z_{\alpha}(\zeta,\eta)\big]$ (see \S~\ref{sec:old} for
the precise definitions)
are also in $[E(\Phi,I),E(\Phi,J)]$. This is exactly the main new 
calculation in the present paper, the rest was either known before, 
or contained in \cite{RNZ-generation}.

Actually, our Theorem~\ref{The:1} allows also to unrelativise the
birelative standard commutator formula, established in this context 
by You Hong \cite{you92}, via level calculations, and by ourselves 
\cite{RNZ2} via a version of relative localisation. Namely, let 
$\rho_I:R\map R/I$ be the reduction modulo $I$. By functoriality,
it defines the group homomorphism $\rho_I:G(\Phi,R)\map G(\Phi,R/I)$.
The kernel of $\rho_I$ is denoted by $G(\Phi,R,I)$ and is called
the \emph{principal congruence subgroup\/} of $G(\Phi,R)$ of level $I$.
In turn, the full pre-image of the centre of $G(\Phi,R/I)$ with
respect to the reduction homomorphism $\rho_{I}$ is called the
\emph{full congruence subgroup\/} of level $I$, and
is denoted by $C(\Phi,R,I)$. Now, the the {\it birelative standard 
commutator formula\/}, see \cite{RNZ2} , Theorem~1, can be stated 
as follows.

\begin{Oldie}\label{Oldie:4}
Let\/ $\Phi$ be a reduced irreducible root system of rank $\ge 2$.
Further, let\/ $R$ be a commutative ring, and\/ $I,J\trianglelefteq R$
be two ideals of\/ $R$. Then
$$ [E(\Phi,R,I),C(\Phi,R,J)]=[E(\Phi,R,I),E(\Phi,R,J)]. $$
\end{Oldie}

Now, Theorems~\ref{The:1} and~\ref{Oldie:4} immediately imply
the following result.

\begin{The}\label{The:2}
Let $\Phi$ be a reduced irreducible root system of rank $\ge 2$
and let $I,J$ be two ideals of a commutative ring $R$. Then
$$ [E(\Phi,I),C(\Phi,R,J)]=[E(\Phi,I),E(\Phi,J)]. $$
\end{The}
\begin{proof} Indeed, one has
\begin{multline*} 
[E(\Phi,I),E(\Phi,J)]\le [E(\Phi,I),C(\Phi,R,J)]\le\\ 
[E(\Phi,R,I),C(\Phi,R,J)]=[E(\Phi, R, I),E(\Phi, R, J)], 
\end{multline*} 
 \noindent
 where the first two inclusions is obvious, whereas the last equality
 is Theorem~\ref{Oldie:4}. On the other hand, the left hand side
 equals the right hand side by Theorem~\ref{The:1}.
\end{proof}

Setting $I=J$ in Theorem~\ref{The:2}, we get the following freakish corollary.

\begin{The}\label{The:3}
Let $\Phi$ be a reduced irreducible root system of rank $\ge 2$
and let $I$ be an ideal of a commutative ring $R$. Then 
$E(\Phi, I)$ is normal in $C(\Phi,R,I)$.
\end{The}

For the special case of $G=\GL(n,R)$ the above Theorems~\ref{The:1} 
and \ref{The:2} were first verified by the first named author in 
\cite{NV18}, while Theorem~\ref{The:3} in that case was proven 
already in \cite{Nica}. However, in \cite{NV18} the proof 
proceeded differently. First, Theorem~\ref{The:2} was derived from
Theorems~\ref{Oldie:1} and~\ref{Oldie:2} by the same birelative
version of decomposition of unipotents 
\cite{Stepanov_Vavilov_decomposition} that was already used in 
\cite{NVAS} to establish the respective special case of 
Theorem~\ref{Oldie:4}.
Then, Theorem~\ref{The:1} was derived as a corollary of 
Theorems~\ref{The:2} and~\ref{Oldie:4}. Thereupon, the second 
author immediately 
suggested that per-case one could achieve the same directly, 
by looking at the elementary generators in \cite{RNZ-generation}, 
Theorem~1.3. This is exactly what we accomplish in the present paper.
Technically, the proofs are not ticklish, the main difficulty was to 
convince ourselves that Theorems~\ref{The:1}--\ref{The:3} can 
be true as stated!

\par
The paper is organised as follows. In \S~\ref{sec:old} we recall
notation and some background facts that will be used in our 
proofs. Also, there we recall Theorem 1.3 of 
\cite{RNZ-generation} and reduce the proof of Theorem~\ref{The:1} 
to a calculation in groups of rank 2. The technical core of the paper 
is \S~\ref{sec:proof}, where we consecutively verify our Main 
Lemma for types $\A_2$, $\C_2$ (which is the most difficult case), 
and $\G_2$. After that in \S~\ref{sec:long} we establish another 
related result, generation of $E(\Phi,R,I)$ by long root elements.
Finally, in \S~\ref{sec:final} we mention some further
related results and applications.


\section{Notation and preliminary facts}\label{sec:old}

To make this paper independent of \cite{RNZ-generation} here we 
recall basic notation and the requisite facts, which will be used in 
our proofs. For more background information on Chevalley groups 
over rings, see \cite{NV91,vavplot,RN1} and references therein.

\subsection{Notation}
Let $G$ be a group. For any $x,y\in G$,  ${}^xy=xyx^{-1}$  denotes
the left $x$-conjugate of $y$. As usual, $[x,y]=xyx^{-1}y^{-1}$
denotes the [left normed] commutator of $x$ and $y$. We shall make
constant use of the obvious commutator identities, such as $[x,yz]=[x,y]\cdot{}^y[x,z]$ or $[xy,z]={}^x[y,z]\cdot[x,z]$,
usually without any specific reference.

As in the introduction, we denote by $x_{\a}(\xi)$, $\a\in\Phi$, $\xi\in R$, 
the elementary generators of the (absolute) elementary Chevalley subgroup 
$E(\Phi,R)$. For a root $\a\in\Phi$ we denote by $X_{\a}$ the
corresponding [elementary] root subgroup 
$X_{\a}=\big\{x_{\a}(\xi)\mid\xi\in R\big\}$. Recall that any conjugate
${}^gx_{\a}(\xi)$ of an elementary root unipotent, where $g\in G(\Phi,R)$
is called {\it root element\/} or {\it root unipotent\/}, it is called {\it long\/} 
or {\it short\/}, depending on whether the root $\a$ itself is long or short.

Let, as in the introduction, $I$ be an ideal of $R$, We denote by
$X_{\a}(I)$ the intersection of $X_{\a}$ with the principal congruence 
subgroup $G(\Phi,R,I)$. Clearly, $X_{\a}(I)$ consists of all elementary 
root elements $x_{\a}(\xi)$, $\a\in\Phi$, $\xi\in I$, of level $I$:
$$ X_{\a}(I)=\big\{x_{\a}(\xi)\mid\xi\in I\big\}. $$
\noindent 
By definition, $E(\Phi,I)$ is generated by $X_{\a}(I)$, for all roots 
$\a\in\Phi$. The same subgroups generate $E(\Phi,R,I)$ as a normal
subgroup of the absolute elementary group $E(\Phi,R)$.
Generators of $E(\Phi,R,I)$ {\it as a group\/} are recalled in the next 
subsection.

All results of the present paper are based on the Steinberg 
relations among the elementary generators, which will be repeatedly 
used without any specific reference. Especially important for us is
the Chevalley commutator formula 
$$ [x_{\a}(\xi),x_{\beta}(\zeta)]=\prod_{i\a+j\beta\in \Phi}
x_{i\a+j\beta}(N_{\a\beta ij}\xi^i\zeta^j), $$
\noindent
where $\a\not=-\beta$ and $N_{\a\beta ij}$ are the structure constants
which do not depend on $\xi$ and $\zeta$. However, for $\Phi=\G_2$
they may depend on the order of the roots in the product on the right
hand side. See \cite{carter,stein2,steinberg,vavplot} for more
details regarding the structure constants $N_{\a\beta ij}$.

\subsection{Generation of mixed commutator subgroups}
We shall extensively use the  two following generation theorems. The
first one is a classical result by Michael Stein \cite{stein2}, Jacques Tits 
\cite{tits} and Leonid Vaserstein \cite{vaser86}. The second one is the
main result of \cite{RNZ3}, Theorem~1.3.

\begin{Oldie}\label{Oldie:5}
Let $\Phi$ be a reduced irreducible root system of rank $\ge 2$
and let $I$ be an ideal of a commutative ring $R$. Then
as a group\/ $E(\Phi,R,I)$ is generated by the elements of
the form
$$ z_{\alpha}(\xi,\eta)=
x_{-\alpha}(\eta)x_{\alpha}(\xi)x_{-\alpha}(-\eta), $$
\noindent
where\/ $\xi\in I$, $\eta\in R$, and $\alpha\in\Phi$.
\end{Oldie}

\begin{Oldie}\label{Oldie:6}
Let $\Phi$ be a reduced irreducible root system of rank $\ge 2$.
In the cases\/ $\Phi=\C_2,\G_2$ assume that $R$ does not have residue
fields\/ ${\Bbb F}_{\!2}$ of\/ $2$ elements and in the
case\/ $\Phi=\C_l$, $l\ge 2$, assume additionally that any\/
$\theta\in R$ is contained in the ideal\/ $\theta^2R+2\theta R$.
\par
Further, let $I$ and $J$ be two ideals of a commutative ring $R$.
Then the mixed commutator subgroup
$\big[E(\Phi,R,I),E(\Phi,R,J)\big]$ is generated as a group by
the elements of the form
\par\smallskip
$\bullet$ $\big[x_{\alpha}(\xi),z_{\alpha}(\zeta,\eta)\big]$,
\par\smallskip
$\bullet$ $\big[x_{\alpha}(\xi),x_{-\alpha}(\zeta)\big]$,
\par\smallskip
$\bullet$ $z_{\alpha}(\xi\zeta,\eta)$,
\par\smallskip\noindent
where in all cases $\alpha\in\Phi$, $\xi\in I$, $\zeta\in J$,
$\eta\in R$.
\end{Oldie}
Now, the generators of second type belong to $[E(\Phi,I),E(\Phi,J)]$
by the very definition. Generators of the third type belong to 
$[E(\Phi,I),E(\Phi,J)]$ by Theorem~\ref{Oldie:1}. Thus, 
Theorem~\ref{Oldie:6} implies that to prove Theorem~\ref{The:1}
it suffices to establish the following result.
\begin{Mainlem}
Let $\Phi$ be a reduced irreducible root system of rank $\ge 2$
and let $I,J$ be two ideals of a commutative ring $R$. Then
$$ \big[x_{\alpha}(\xi),z_{\alpha}(\zeta,\eta)\big]\in [E(\Phi,I),E(\Phi,J)], $$
\noindent
for all
$\alpha\in\Phi$, $\xi\in I$, $\zeta\in J$, $\eta\in R$.
\end{Mainlem}
Obviously, the proof of the Main Lemma immediately reduces to 
rank 2 systems. Thus, we only have to verify it for groups of types 
$\A_2$, $\B_2$ and $\G_2$. For $\Phi=\A_2$ we reproduce an
authentic calculation at the level of individual elementary generators,
with actual signs (which in this case is an adaptation of an argument 
from \cite{NV18}). We could do the same also for $\Phi=\C_2,\G_2$,
and this was, as a matter of fact, how we originally verified it.
However, to make the text more readable, we prefer the following
shortcut. Since we already know Theorems~\ref{Oldie:1} and 
\ref{Oldie:2}, we can perform calculations modulo the subgroups
$E(n,R,IJ)$ and $[E(n,I),E(n,J)]$. In turn, in many cases the easiest 
way to verify that some commutators fall into these subgroups, 
is Levi decomposition, which we now recall, in a slightly more
precise form, than the one used in \cite{RNZ-generation}.

\subsection{Parabolic subgroups}
An important part in the proof of Main Lemma for $\Phi=\C_2,\G_2$,
is played by Levi decomposition for [elementary] parabolic subgroups.
Classically, it asserts that any parabolic subgroup $P$ of
$G(\Phi,R)$ can be expressed as the semi-direct product
$P=L_P\rightthreetimes U_P$ of its unipotent radical
$U\unlhd P$ and a Levi subgroup $L_P\le P$. However, as in
\cite{RNZ-generation} we do not have to recall the general case. 
\par\smallskip
$\bullet$ Since we calculate inside $E(n,R)$, we can limit
ourselves to the {\it elementary\/} parabolic subgroups,
spanned by some root subgroups $X_{\a}$.
\par\smallskip
$\bullet$ Since we can choose the order on $\Phi$ arbitrarily,
we can always assume that $\a$ is fundamental and, thus, limit 
ourselves to {\it standard\/} parabolic subgroups.
\par\smallskip
$\bullet$ Since the proof of Main Lemma reduces to groups
of rank 2, we could only consider rank 1 parabolic subgroups,
which {\it in this case\/} are maximal parabolic subgroups.
\par\smallskip
Thus, we consider only {\it elementary\/} rank 1 parabolics, which 
are defined as follows. Namely, we fix an order on $\Phi$, and
let $\Phi^+$ and $\Phi^-$ be the corresponding sets of positive 
and negative roots, respectively. Further, let $\Pi=\{\a_1,\ldots,\a_l\}$
be the corresponding fundamental system. For any $r$, $1\le r\le l$, 
and define the $r$-th rank 1 {\it elementary\/} parabolic subgroup as
$$ P_{\a_r}=\langle U,X_{\a_r}\rangle\le E(\Phi,R). $$
\noindent
Here $U= \prod X_{\a}$, $\a\in\Phi^+$, is the unipotent radical of
the standard Borel subgroup $B$. Then the unipotent radical of $P_{\a_r}$
has the form
$$ U_{\a_r}=\prod X_{\a},\quad \a\in\Phi^+,\  \a\neq\a_r, $$
\noindent
whereas $L_{\a_r}=\langle X_{\a_r},X_{-\a_r}\rangle$ is the [standard] 
Levi subgroup of $P_r$. Clearly, $L_{\a_r}$ is isomorphic to the elementary 
subgroup $E(2,R)$ in $\SL(2,R)$, or to its projectivised version 
$\PE(2,R)$ in $\PGL(2,R)$. In the sequel we usually (but not always!)
abbreviate $P_{\a_r},U_{\a_r},L_{\a_r}$, etc., to $P_r,U_r,L_r$, etc.
\par
Levi decomposition (which in the case of elementary parabolics 
immediately follows from
the Chevalley commutator formula) asserts that the group $P_r$ 
is the semi-direct product $P_r=L_r\rightthreetimes U_r$ of 
$U_r\unlhd P_r$ and $L_r\le P_r$. The most important part is the 
[obvious] claim is that $U_r$ is normal in $P_r$.
\par
Simultaneously with $P_r$ one considers also the opposite parabolic 
subgroup $P_r^-$ defined as
$$ P_r=\langle U^-,X_{-\a_r}\rangle\le E(\Phi,R). $$
\noindent
Here $U^-= \prod X_{\a}$, $\a\in\Phi^-$, is the unipotent radical of
the Borel subgroup $B^-$ opposite to the standard one. Clearly,
$P_r$ and $P_r^-$ share the common [standard] Levi subgroup $L_r$, 
whereas the unipotent radical $U_r^-$ of $P_r^-$ is opposite to that 
of $P_r$, and has the form
$$ U_r^-=\prod X_{\a},\quad \a\in\Phi^-,\  \a\neq-\a_r. $$
\noindent
Now, Levi decomposition takes the form
$P_r^-=L_r\rightthreetimes U_r^-$ with $U_r^-\unlhd P_r^-$. In other
words, $U_r$ and $U_r^-$  are both normalised by $L_r$. 

Actually, we need a slightly more precise form of this last statement.
Namely, let $I$ be an ideal of $R$. Denote by $L_r(I)$ the principal
congruence subgroup of level $I$ in $L_r$ and by $U_r(I)$ and 
$U_r^-(I)$ the respective intersections of $U_r$ and $U_r^-$ with
$G(\Phi,R,I)$ --- or, what is the same, with $E(\Phi,R,I)$:
$$ U_r(I)=U_r\cap E(\Phi,R,I),\qquad U_r^-(I)=U_r^-\cap E(\Phi,R,I). $$
\noindent
Obviously, $U_r(I),U_r^-(I)\le E(\Phi,I)$ are normalised by $L_r$. The
following fact will be repeatedly used in the proof of the Main Lemma.
\begin{Lemma}
Let $I$ and $J$ be two ideals of $R$. Then 
$$ [L_r(I),U_r(J)]\le U_r(IJ),\qquad [L_r(I),U_r^-(J)]\le U_r^-(IJ). $$
\noindent
In particular, both commutators are contained in $E(\Phi,IJ)\le E(\Phi,R,IJ)$.
\end{Lemma}


\section{Proof of Main Lemma}\label{sec:proof}

In this section we prove the Main Lemma --- and thus also 
Theorems~\ref{The:1}--\ref{The:3}. Let, as above, 
$x=[x_{\a}(\xi),z_{\a}(\zeta,\eta)]$, where $\xi\in I$, $\zeta\in J$, 
$\eta\in R$. We divide the proof into four cases: 
\par\smallskip
i) $\a$ can be embedded in a root subsystem of type $\A_2$.
This proves the Main Lemma for simply laced Chevalley 
groups, and for the Chevalley group of type $\F_4$. It also proves 
the inclusion in the Main Lemma for {\it short\/} roots in Chevalley 
groups of type $\C_l$, $l\ge 3$, for {\it long\/} roots in Chevalley 
groups of type $\B_l$, $l\ge 3$, and for {\it long\/} roots in the 
Chevalley group of type $\G_2$.
\par\smallskip
ii) $\a$ can be embedded in a root subsystem of type $\C_2$ 
as a {\it long\/} root. This proves the Main Lemma for Chevalley 
groups of type $\C_l$, $l\ge 3$. 
\par\smallskip
iii) $\a$ can be embedded in a root subsystem of type $\C_2$
as a {\it short\/} root. This proves the Main Lemma for Chevalley 
groups of type $\B_l$, $l\ge 3$, and finishes the proof for the
case $\C_2$. 
\par\smallskip
iv) $\a$ can be embedded in a root subsystem of type $\G_2$ 
as a {\it short\/} root. This proves the Main Lemma for the last
remaining case, the group of type $\G_2$. 
\par\smallskip
For the first case, we reproduce an actual computation at the level 
of root elements that ultimately could be refined to an explicit formula
expressing $x=[x_{\a}(\xi),z_{\a}(\zeta,\eta)]$ as a product of 
conjugates of the commutators of the form 
$[x_{\g}(\epsilon),x_{\delta}(\theta)]$, for some roots 
$\gamma,\delta\in\Phi$ and some $\epsilon\in I$, $\theta\in J$.
This argument is a transcript of the initial argument from \cite{NV18}, 
which corresponds to the the first --- difficult! --- item in the proof of 
\cite{NV18}, Theorem~1. 
\par\smallskip
$\bullet$ 
First, assume that $\alpha$ can be embedded in a root system of type 
$\A_2$. 
We wish to prove that $[x_{\a}(\xi),z_{\a}(\zeta,\eta)]\in [E(\A_2,I),E(\A_2,J)]$. 
Indeed, in this case there exist roots $\b,\g\in\Phi$, of the same length 
as $\a$ such that $\a=\b+\g$ and $N_{\b\g11}=1$.
Then
$$ x=[x_{\a}(\xi),z_{\a}(\zeta,\eta)]=
x_{\a}(\xi)\cdot{}^{z_{\a}(\zeta,\eta)}x_{\a}(-\xi)= 
x_{\a}(\xi)\cdot{}^{z_{\a}(\zeta,\eta)}[x_{\b}(1),x_{\g}(-\xi)]. $$
\noindent
Thus, 
\begin{multline*}
x=x_{\a}(\xi)\cdot[{}^{z_{\a}(\zeta,\eta)}x_{\b}(1),
{}^{z_{\a}(\zeta,\eta)}x_{\g}(-\xi)]=\\
x_{\a}(\xi)\cdot [x_{\b}(1-\zeta\eta)x_{-\g}(-\eta\zeta\eta),
x_{-\b}(-\xi\eta\zeta\eta)x_{\g}(-\xi(1-\eta\zeta))]=\\
x_{\a}(\xi)\cdot[x_{\b}(1)y,x_{\g}(-\xi)z],
\end{multline*}
\noindent
where
$$ y=x_{\b}(-\zeta\eta)x_{\g}(-\eta\zeta\eta)\in E(\A_2,J),\quad 
z=x_{-\b}(-\xi\eta\zeta\eta)x_{\g}(\xi\eta\zeta)\in E(\A_2,IJ). $$
\noindent
Since $x_{\g}(\xi)\in E(\A_2,I)$, the second factor of the above commutator
belongs to $E(\A_2,I)$. Thus,
$$ [x_{\b}(1)y,x_{\g}(-\xi)z]={}^{x_{\b}(1)}[y,x_{\g}(-\xi)z]
\cdot[x_{\b}(1),x_{\g}(-\xi)z]. $$
\par
Now the first commutator in the right hand side belongs to 
$[E(\A_2,I),E(\A_2,J)]$, which is normal in $E(\A_2,R)$, so that the
conjugation by $x_{\b}(1)$ still leaves us there.
\par
On the other hand, the second commutator equals
$$ [x_{\b}(1),x_{\g}(-\xi)]\cdot{}^{x_{\g}(-\xi)}[x_{\b}(1),z], $$
\noindent
The second commutator in the last expression belongs to $E(\A_2,R,IJ)$, 
and remains there after elementary conjugations, whereas the first 
commutator equals $x_{\a}(-\xi)$.
\par
Summarising the above, we see that 
$$ x\in x_{\a}(\xi)[E(\A_2,I),E(\A_2,J)]x_{\a}(-\xi)\cdot 
E(\A_2,R,IJ)\le [E(\A_2,I),E(\A_2,J)], $$
\noindent
as claimed.
\par\smallskip
For the three remaining cases, where $\Phi=\C_2$ or $\Phi=\G_2$, 
the idea of proof is similar but its
implementation requires more care, because of the more complicated 
form of the Chevalley commutator formula. In these cases too we 
{\it could\/} come up with explicit formulas, but to restrain the length,
we prefer to repeatedly invoke the above Lemma on unipotent radicals, 
and Theorems~\ref{Oldie:1}, \ref{Oldie:2}. In other words, all 
calculations are performed modulo $[E(\Phi,I),E(\Phi,J)]$, which is
already normal in $E(\Phi,R)$. At the moment we discover that a 
certain factor falls into $E(\Phi,R,IJ)$ or into $[E(\Phi,I),E(\Phi,J)]$
itself, we immediately loose interest to the explicit form of this factor.
\par
The argument proceeds as follows. When $\a$ is {\it short\/}, we express 
it in the form $\a=\b+\g$, where $\b$ is long and $\g$ is short. Similarly,
when $\a$ is {\it long\/}, we express it in the form $\a=\b+2\g$, with the
same $\b,\g$ as above. Since we are only looking at {\it one\/} instance
of the Chevalley commutator formula, the parametrisation of the
corresponding root subgroups can be chosen in such a way that all
the resulting structure constants are positive, so that the formula
takes the form
$$ [x_{\b}(\xi),x_{\g}(\theta)]=x_{\b+\g}(\xi\theta)x_{\b+2\g}(\xi\theta^2) $$
\noindent
in the case of $\Phi=\C_2$ and the form 
$$ [x_{\b}(\xi),x_{\g}(\theta)]=x_{\b+\g}(\xi\theta)x_{\b+2\g}(\xi\theta^2) 
x_{\b+3\g}(\xi\theta^3)x_{2\b+3\g}(2\xi^2\theta^3) $$
\noindent
in the case of $\Phi=\G_2$, see \cite{carter, steinberg} or \cite{vavplot}
and references there.
\par
As above, we rewrite $x$ as $x=x_{\a}(\xi)\cdot{}^{z_{\a}(\zeta,\eta)}x_{\a}(-\xi)$
and plug in the expression of $x_{\a}(-\xi)$ as the commutator
$[x_{\b}(\xi),x_{\g}(1)]^{-1}=[x_{\g}(1),x_{\b}(\xi)]$ times the tail consisting 
of the remaining factors $x_{\delta}(\eta)$ from the above instances
of the Chevalley commutator formula, which, up to sign are equal
to $\xi$ or $2\xi^2$ and in any case belong to $E(\Phi,I)$. By
the above Lemma the conjugates ${}^{z_{\a}(\zeta,\eta)}x_{\delta}(\eta)$
of the remaining factors are congruent to these factors themselves,
modulo $E(\Phi,R,IJ)$. As in the first case, this again leaves us with 
analysis of the commutator ${}^{z_{\a}(\zeta,\eta)}[x_{\g}(1),x_{\b}(\xi)]$,
slightly different between cases, due to disparate configurations of roots.
Anyway, in each case the result will be that modulo elementary 
conjugations and factors that cancel with $x_{\a}(\xi)$, or with the
outstanding factors coming from the Chevalley commutator formula,
the relevant part of the commutator falls into $[E(\Phi,I),E(\Phi,J)]$.
\par
Now we pass to the case by case analysis.
\par\smallskip
$\bullet$ First, assume that $\a$ can be embedded into $\C_2$ 
as a {\it long\/} root. In this case there exist a long root $\b$ and 
a short root $\g$ such that $\a=\b+2\g$ and choosing the signs
in the Chevalley commutator formula as above, we can write
$x_{\a}(-\xi)=[x_{\g}(1),x_{\b}(\xi)]x_{\b+\g}(\xi)$. Plugging this 
into the expression for $x$, we get
$$ x=x_{\a}(\xi)\cdot
{}^{z_{\a}(\zeta,\eta)}[x_{\g}(1),x_{\b}(\xi)]\cdot
{}^{z_{\a}(\zeta,\eta)}x_{\b+\g}(\xi). $$
\par
As we know from Lemma, 
$$ {}^{z_{\a}(\zeta,\eta)}x_{\b+\g}(\xi)\equiv x_{\b+\g}(\xi)
\pmod{E(\C_2,IJ)}, $$
\noindent
so that ${}^{z_{\a}(\zeta,\eta)}x_{\b+\g}(\xi)$ can be rewritten in the
form $x_{\b+\g}(\xi)z$, for some $E(\C_2,IJ)$.
\par
Next, we look at the second factor. Clearly,
$$ y={}^{z_{\a}(\zeta,\eta)}[x_{\g}(1),x_{\b}(\xi)]=
[{}^{z_{\a}(\zeta,\eta)}x_{\g}(1),{}^{z_{\a}(\zeta,\eta)}x_{\b}(\xi)]=
[{}^{z_{\a}(\zeta,\eta)}x_{\g}(1),x_{\b}(\xi)]. $$
\noindent
As we know from Lemma, ${}^{z_{\a}(\zeta,\eta)}x_{\g}(1)\equiv x_{\g}(1)
\pmod{E(\C_2,J)}$. Rewriting ${}^{z_{\a}(\zeta,\eta)}x_{\g}(1)$
in the form ${}^{z_{\a}(\zeta,\eta)}x_{\g}(1)=x_{\g}(1)w$, for some
$w\in E(\C_2,J)$, we get
$$ y= [x_{\g}(1)w,x_{\b}(\xi)]={}^{x_{\g}(1)}[w,x_{\b}(\xi)]
\cdot [x_{\g}(1),x_{\b}(\xi)], $$
\noindent
where the first commutator belongs to $[E(\C_2,I),E(\C_2,J)]$, 
and stays there after elementary conjugation. 
\par
Combining the above, and expanding $[x_{\g}(1),x_{\b}(\xi)]$ by
the Chevalley commutator formula, we see that
$$ x=x_{\a}(\xi)\cdot
{}^{x_{\g}(1)}[x_{\b}(\xi),w]\cdot
x_{\a}(-\xi)x_{\b+\g}(-\xi)\cdot 
x_{\b+\g}(\xi)z\in [E(\C_2,I),E(\C_2,J)], $$
\noindent
as claimed.
\par\smallskip
$\bullet$ Next, assume that $\a$ can be embedded in $\C_2$ as a 
{\it short\/} root. Choose $\beta$ and $\gamma$ such that $\a=\beta+\gamma$,
while $N_{\beta\gamma 11}=N_{\beta\gamma 12}=1$. Then, 
clearly, $x_{\a}(-\xi)$ can be expressed as 
$x_{\a}(-\xi)=[x_{\g}(1),x_{\b}(\xi)]x_{\b+2\g}(\xi)$. Thus,
$$ x=
x_{\a}(\xi)\cdot
{}^{z_{\a}(\zeta,\eta)}[x_{\g}(1),x_{\b}(\xi)]\cdot
{}^{z_{\a}(\zeta,\eta)}x_{\b+2\g}(\xi). $$
\par
Again by Lemma ${}^{z_{\a}(\zeta,\eta)}x_{\b+2\g}(\xi)=x_{\b+2\g}(\xi)z$
for some $z\in E(\C_2,IJ)$. 

Looking at the second factor, we see that 
$$ y={}^{z_{\a}(\zeta,\eta)}[x_{\g}(1),x_{\b}(\xi)]=
[{}^{z_{\a}(\zeta,\eta)}x_{\g}(1),{}^{z_{\a}(\zeta,\eta)}x_{\b}(\xi)]. $$
\noindent
Now, rewriting ${}^{z_{\a}(\zeta,\eta)}x_{\g}(1)$
in the form ${}^{z_{\a}(\zeta,\eta)}x_{\g}(1)=x_{\g}(1)w$, for some
$w\in E(\C_2,J)$, we get
$$ y= [x_{\g}(1)w,{}^{z_{\a}(\zeta,\eta)}x_{\b}(\xi)]=
{}^{x_{\g}(1)}[w,{}^{z_{\a}(\zeta,\eta)}x_{\b}(\xi)]
\cdot [x_{\g}(1),{}^{z_{\a}(\zeta,\eta)}x_{\b}(\xi)]. $$
\noindent
The first commutator here belongs to $[E(\C_2,I),E(\C_2,J)]$, 
and stays there after elementary conjugation. On the other hand, 
applying Lemma once more we see that ${}^{z_{\a}(\zeta,\eta)}x_{\b}(\xi)$
can be rewritten in the form $x_{\b}(\xi)v$, for some $v\in E(\C_2,IJ)$.
This means that 
$$  [x_{\g}(1),{}^{z_{\a}(\zeta,\eta)}x_{\b}(\xi)]= [x_{\g}(1),x_{\b}(\xi)v]=
[x_{\g}(1),x_{\b}(\xi)]\cdot{}^{x_{\b}(\xi)} [x_{\g}(1),v]. $$
\noindent
The second commutator here belongs to $E(\C_2,R,IJ)$ and stays
there after the elementary conjugation.

Combining the above, and expanding $[x_{\g}(1),x_{\b}(\xi)]$ by
the Chevalley commutator formula, we see that
$$ x=x_{\a}(\xi)\cdot 
{}^{x_{\g}(1)}[w,{}^{z_{\a}(\zeta,\eta)}x_{\b}(\xi)]\cdot
x_{\a}(-\xi)x_{\b+2\g}(-\xi)\cdot 
{}^{x_{\b}(\xi)}[x_{\g}(1),v]\cdot
x_{\b+2\g}(\xi)z. $$
\noindent
Here the first commutator belongs to $[E(\C_2,I),E(\C_2,J)]$, and
stays there after elementary conjugation, whereas the second 
commutator belongs to $E(\C_2,R,IJ)$, and stays there after 
elementary conjugation, while the outstanding factor $z$ already 
belongs to $E(\C_2,IJ)$, as claimed.
\par\smallskip
$\bullet$ This leaves us with the analysis of the case, when
$\a$ is a {\it short\/} root of $\Phi=\G_2$. Choose a long root $\b$
and a short root $\gamma$ such that $\a=\b+\g$, and the structure
constants are as above, $N_{\b\g11}=N_{\b\g12}=N_{\b\g13}=1$,
$N_{\b\g23}=2$. Then $x_{\a}(-\xi)$ can be expressed as 
$$ x_{\a}(-\xi)=[x_{\g}(1),x_{\b}(\xi)]\cdot
x_{\b+2\g}(\xi)x_{\b+3\g}(\xi)x_{2\b+3\g}(2\xi^2). $$
\noindent
Plugging this into the expression for $x$, we get
$$ x=
x_{\a}(\xi)\cdot
{}^{z_{\a}(\zeta,\eta)}[x_{\g}(1),x_{\b}(\xi)]\cdot
{\vphantom{\big(}}^{z_{\a}(\zeta,\eta)}\big(x_{\b+2\g}(\xi)
x_{\b+3\g}(\xi)x_{2\b+3\g}(2\xi^2)\big). $$
\par
By the same token, we see that the last factor belongs to 
the unipotent radical of the parabolic subgroup $P_{\a}$, and, thus,
by Lemma can be rewritten as
$$ {\vphantom{\big(}}^{z_{\a}(\zeta,\eta)}\big(x_{\b+2\g}(\xi)
x_{\b+3\g}(\xi)x_{2\b+3\g}(2\xi^2)\big)=
x_{\b+2\g}(\xi)x_{\b+3\g}(\xi)x_{2\b+3\g}(2\xi^2)\cdot z, $$
\noindent
for some $z\in E(\G_2,IJ)$.
\noindent
Now, repeating exactly the same calculation as in the previous
case, we see that the second factor in the above expression
for $x$ has the form
$$ {}^{x_{\g}(1)}[w,{}^{z_{\a}(\zeta,\eta)}x_{\b}(\xi)]\cdot
[x_{\g}(1),x_{\b}(\xi)]\cdot{}^{x_{\b}(\xi)} [x_{\g}(1),v], $$
\noindent
for some $w\in E(\G_2,J)$ and $v\in E(\G_2,IJ)$.

Combining the above, and once more expanding 
$[x_{\g}(1),x_{\b}(\xi)]$ by the Chevalley commutator formula, 
we see that
\begin{multline*}
x=x_{\a}(\xi)\cdot 
{}^{x_{\g}(1)}[w,{}^{z_{\a}(\zeta,\eta)}x_{\b}(\xi)]\cdot
x_{\a}(-\xi)x_{\b+2\g}(-\xi)x_{\b+3\g}(-\xi)x_{2\b+3\g}(-2\xi^2)\cdot \\
{}^{x_{\b}(\xi)}[x_{\g}(1),v]\cdot
x_{\b+2\g}(\xi)x_{\b+3\g}(\xi)x_{2\b+3\g}(2\xi^2)\cdot z
\end{multline*}
(recall that for the above choice of structure constants 
$[x_{a}(\xi),x_{\b+2\g}(\eta)]=x_{2\b+3\g}(3\xi\eta)$, whereas  
root elements corresponding to the roots $\b+2\g,\b+3\g,2\b+3\g$
commute). Here the first commutator belongs to $[E(\G_2,I),E(\G_2,J)]$, 
and stays there after elementary conjugation, the second 
commutator belongs to $E(\G_2,R,IJ)$, and stays there after 
elementary conjugation, while the outstanding factor $z$ already 
belongs to $E(\G_2,IJ)$, as claimed.
\par\smallskip
This completes the proof of Main Lemma, and thus also of all 
other new results stated in the Introduction.


\section{Generation of elementary subgroups by long root unipotents}
\label{sec:long}

In this section we prove another result pertaining to generation 
of relative elementary subgroups, closely related to the contents of
\cite{RNZ-generation}, and the present paper. Namely, we prove that
$E(\Phi,R,I)$ is generated by {\it long\/} root unipotents. There is 
no doubt that this result is known for several decades, and is 
immediately obvious to experts. However, we are not aware of any 
explicit source. 

The purpose to include this result here is two-fold. First, we need it
for future reference in the work by the first named author on
the width of root type unipotents in $\Sp(2l,R)$ and in $G(\F_4,R)$ 
with respect to the elementaries. Second, it would be very 
interesting to understand, what this result means for the generation
of relative commutator subgroups $[E(\Phi,R,I),E(\Phi,R,J)]$ and 
whether one could accordingly reduce their sets of generators 
obtained in \cite{RNZ-generation}, Theorem 1.3.

\begin{The}\label{The:5}
Let $\rk(\Phi)\ge 2$, for $\Phi=\G_2$ assume additionally
that $R$ does not have residue field $\GF{2}$ of $2$ elements. Then
for any ideal $I\unlhd R\label{sec:1}$ the relative elementary group $E(\Phi,R,I)$
is generated by long root elements.
\end{The}
\begin{proof}
 For $\Phi=\A_l,\D_l,E_l$ there is nothing to prove. Thus,
let $\Phi=\B_l,\C_l,F_4$ or $\G_2$. It suffices to prove that any
elementary short root element $x_{\b}(\xi)$, where $\b\in\Phi_s$
and $\xi\in I$, is a product of long root elements
$x_1,\ldots,x_m\in E(\Phi,R,I)$. If this is the case, then for any
$g\in E(\Phi,R)$ its conjugate
${}^g{x_{\b}(\xi)}={}^g{x_1}\cdot\ldots\cdot{}^g{x_m}$
is also a product of long root elements from $E(\Phi,R,I)$.
\par
First, let $\Phi\neq\G_2$. Then there exists a long root $\a$ and
a short root $\g$ such that $\b=\a+\g$. Then the root $\a+2\g=\b+\g$
is long and carrying the corresponding factor to the left hand side
in the commutator formula
$$ [x_{\a}(\xi),x_{\g}(1)]=x_{\b}(\pm\xi)x_{\b+\g}(\pm\xi), $$
\noindent
we express $x_{\b}(\pm\xi)$ as the product of three long root
unipotents
$$ x_{\b}(\pm\xi)=x_{\a}(\xi)\big(x_{\g}(1)x_{\a}(-\xi)x_{\g}(-1)\big)
x_{\b+\g}(\mp\xi), $$
\noindent
sitting in $E(\Phi,R,I)$.
\par
On the other hand, for the case $\Phi=\G_2$ there exists a long root
$\a$ and a short root $\g$ such that $\b=\a+2\g$. Then the root
$\a+\g=\b-\g$ is short, whereas the roots $\a+3\g=\b+\g$ and
$2\a+3\g=2\b-\g$ are both long. Plugging in the Chevalley commutator
formula
$$ [x_\a{\eta},x_{\g}(\zeta)]=
x_{\a+\g}(\pm\eta\zeta)x_{\b}(\pm\eta\zeta^2)
x_{\b+\g}(\pm\eta\zeta^3)x_{2\b-\g}(\pm\eta^2\zeta^3) $$
\noindent
first $\eta=\xi\in I$ and $\zeta=\theta\in R$ and then $\eta=\xi\theta$
and $\zeta=1$, for the same $\xi\in I$ and $\theta\in R$, and carrying
over the factors corresponding to the long roots $\b+\g$ and $2\b-\g$
to the left hand side of the resulting commutator formulas, we get
the following expressions. Firstly,
$$ y=x_{\a+\g}(\pm\xi\theta)x_{\b}(\pm\xi\theta^2)=
x_\a(\xi)\big(x_{\g}(\theta)x_\a(-\xi)x_{\g}(-\theta)\big)
x_{\b+\g}(\mp\xi\theta^3)x_{2\b-\g}(\mp\xi^2\theta^3) $$
\noindent
is a product of four long root unipotents sitting in $E(\Phi,R,I)$.
Similarly,
$$ z=x_{\a+\g}(\pm\xi\theta)x_{\b}(\pm\xi\theta)=
x_\a(\xi\theta)\big(x_{\g}(1)x_\a(-\xi\theta)x_{\g}(-1)\big)
x_{\b+\g}(\mp\xi\theta)x_{2\b-\g}(\mp\xi^2\theta^2) $$
\noindent
is a product of four such long root unipotents. Comparing these
equalities, we get an expression
$$ x_{\b}(\pm\xi(\theta^2-\theta))=yz^{-1} $$
\noindent
as a product of not more than six long root elements from
$E(\Phi,R,I)$. Since $R$ does not have residue field of two elements,
the ideal generated by $\theta^2-\theta$, where $\theta\in R$, is not
contained in any maximal ideal, and thus coincides with $R$. This
means that $x_{\b}(\xi)$ is a product of finitely many long root
elements from $E(\Phi,R,I)$.
\end{proof}


\section{Final remarks}\label{sec:final}

The main results of the present paper were completely unexpected
to us, and to several other experts in the structure theory of algebraic
groups over rings, with whom we discussed the subject of the present
paper. Once more, these results highlight the relative commutator 
subgroups $[E(\Phi,R,I),E(\Phi,R,J)]$ as an ubiquitous class of 
subgroups, that occur {\it surprisingly\/} often.

For the general linear group $\GL(n,R)$ these and other concomitant 
birelative groups were first considered in the seminal work of Hyman 
Bass \cite{Bass2}, and then systematically studied by Alec Mason and 
Wilson Stothers \cite{MAS3,Mason74,MAS1,MAS2}. At that stage, the 
standing premise was that $n\ge\sr(R)+1$.

In \cite{NVAS,NVAS2}
the first author and Alexei Stepanov observed that the standard
commutator formula holds for arbitrary commutative rings and in 
\cite{RZ11,RZ12} Roozbeh Hazrat and the second author proposed an approach
based on localisation. As part of that approach, in the linear case they 
found generators of relative commutator groups, which was a starting 
point of the present work.

Later, we together with Roozbeh Hazrat generalised the relative and birelative versions of
localisation, the commutator formulas themselves, and results
on generation of relative commutator groups to unitary 
groups \cite{RNZ1,RNZ3}, and to Chevalley groups 
\cite{RNZ2,RNZ-generation}. Luckily, at that time we were not aware of 
the pioneering work by Hong You \cite{you92}, see the footnote on
page 265 of \cite{RNZ2}.

These results were instrumental
in the work by Alexei Stepanov on the commutator width of Chevalley
groups, see \cite{SV10,Stepanov_nonabelian,Stepanov_universal}.
See also \cite{HPV,AS} for other interesting occurrences of 
the above commutator subgroups in the theory of Chevalley groups.
One can find many further related results, applications and open 
problems in our surveys and conference papers \cite{yoga,portocesareo,yoga2,RNZ4}.

So far, we have not even mentioned another extremely important
line of research, that initially was our main motivations to focus on 
relative commutator subgroups. Namely, the study of subgroups, 
normalised by a relative elementary subgroup, see \cite{RN,yoga,portocesareo,yoga2,RNZ4} for an outline and
further references. We plan to return to this problem in the context 
of Chevalley groups in our next publication.

We are very grateful to 
Roozbeh Hazrat, Andrei Lavrenov and Alexei Stepanov for 
extremely useful discussions. 



\end{document}